\documentclass[11pt,reqno,a4paper]{amsart}
\usepackage[english]{babel}
\usepackage[T1]{fontenc}
\usepackage[utf8]{inputenc}

\usepackage{lineno}
\usepackage{amsfonts}
\usepackage{mathtools}
\usepackage{verbatim}
\usepackage{mathtools}
\usepackage{dsfont}
\usepackage{amsthm}
\usepackage{indentfirst}
\usepackage{amsmath,amssymb}
\usepackage{verbatim}
\usepackage{array}
\usepackage{fancyhdr}
\usepackage{xcolor}
\usepackage{noindentafter}
\usepackage{esint}
\usepackage{mathrsfs}
\usepackage{stmaryrd}
\usepackage{environ}
\usepackage[hidelinks]{hyperref}
\usepackage{amsmath}
\usepackage{enumerate}
\usepackage{xfrac}
\usepackage{thmtools}

\NoIndentAfterCmd{\]}
\NoIndentAfterEnv{equation}
\NoIndentAfterEnv{align*}
\NoIndentAfterEnv{itemize}
\NoIndentAfterEnv{enumerate}

\newcommand{\weak}{\overset{*}{\rightharpoonup}}

\DeclareMathOperator{\id}{id}

\newcommand{\R}{\mathbb{R}}
\newcommand{\Z}{\mathbb{Z}}
\newcommand{\T}{\mathbb{T}}
\newcommand{\E}{\mathds{E}}

\newcommand{\N}{\mathbb{N}}
\newcommand{\D}{\mathbb{D}}

\DeclareMathOperator{\cof}{cof}

\DeclareMathOperator{\dx}{dx}
\DeclareMathOperator{\dy}{dy}
\DeclareMathOperator{\dv}{div}
\DeclareMathOperator{\tr}{tr}

\DeclareMathOperator{\spt}{supp}

\DeclareMathOperator{\Sym}{Sym}

\theoremstyle{plain}
\newtheorem*{NTeo}{Theorem}
\newtheorem*{NProp}{Proposition}
\newtheorem{Teo}{Theorem}

\newtheorem{prop}[Teo]{Proposition}

\newtheorem{Cor}[Teo]{Corollary}
\theoremstyle{definition}

\theoremstyle{remark}
\newtheorem{rem}{Remark}

\NoIndentAfterEnv{Def}
\NoIndentAfterEnv{lemma}
\NoIndentAfterEnv{prop}
\NoIndentAfterEnv{Cor}
\NoIndentAfterEnv{prop}
\NoIndentAfterEnv{rem}
\NoIndentAfterEnv{Rem}
\NoIndentAfterEnv{Teo}
\NoIndentAfterEnv{Theorem}
\NoIndentAfterEnv{DT}

\title{On the upper semicontinuity of a quasiconcave functional}

\author{Luigi De Rosa}
\address{L.D.R.: EPFL SB, Station 8, CH-1015 Lausanne, Switzerland}
\email{luigi.derosa@epfl.ch}

\author{Denis Serre}
\address{D.S.: UMPA, ENS-Lyon, allée d'Italie, 69364 Lyon Cedex 07, France}
\email{denis.serre@ens-lyon.fr}

\author{Riccardo Tione}
\address{R.T.: Institut f\"ur Mathematik, Universit\"at Z\"urich, Winterthurerstrasse 190, CH-8057 Zurich, Switzerland}
\email{riccardo.tione@math.uzh.ch}

\begin{document}

\maketitle

\begin{abstract}
In the recent paper \cite{SER}, the second author proved a divergence-quasiconcavity inequality for the following functional $ \mathbb{D}(A)=\int_{\T^n} \det(A(x))^{\frac{1}{n-1}}\dx$ defined on the space of $p$-summable positive definite matrices with zero divergence. We prove that this implies the weak  upper semicontinuity of the functional $\D(\cdot)$ if and only if $p>\frac{n}{n-1}$.
\end{abstract}
\par
\medskip\noindent
{\sc MSC (2010): 26B25, 39B42, 39B62, 49N60.
\par
}
\section{Introduction}
In this paper, we study the functional 
$$
\D(A)=\int_{\T^n}\det(A(x))^\frac{1}{n-1}\dx
$$
defined on the space
$$
X_p=\{ A\in L^p(\T^n,\Sym^+(n)) \, : \, \dv A\in \mathcal{M}(\T^n, \R^n)\},
$$
where $\T^n$ is the $n$-dimensional torus of $\R^n$, $\Sym^+(n)$ is the space of symmetric $n\times n$ non-negative definite matrices, and $\mathcal{M}(\T^n, \R^n)$ is the space of bounded Radon measures on $\T^n$ with values in $\R^n$.\\

In the recent paper \cite{SER}, the second author proved that the functional is well defined on $ X_1$, meaning that for any $A\in X_1$ the function $\det(A)^\frac{1}{n-1} \in L^1(\R^n/\Gamma)$. More precisely, he proved the following (here we state the divergence-free version since it will be useful for our later discussion, but see also \cite[Theorems 2.2, 2.3, 2.4]{SER} for more general results and \cite{SER2}[Theorem 2.1] for an improvement)

\begin{NTeo} Let the divergence-free, non-negative definite matrix field $x \mapsto A(x)$ be $\Gamma$-periodic, with $A \in L^1(\R^n/ \Gamma)$. Then
\[
\det(A)^\frac{1}{n-1} \in L^1(\R^n/\Gamma)
\]
and there holds
\begin{equation}\label{quasiconcave}
\fint_{\R^n/\Gamma}\det(A(x))^{\frac{1}{n - 1}}\dx \le \det\left(\fint_{\R^n/\Gamma} A(x)\dx\right)^{\frac{1}{n - 1}}.
\end{equation}
\end{NTeo}
In the previous result, $\Gamma \subset \R^n$ is a lattice. For simplicity, in the sequel we will just consider $\Gamma = \Z^n$, hence $\R^n/\Gamma = \T^n$. 
Note that inequality \eqref{quasiconcave} is a generalized Jensen inequality for non-concave functions, and can be viewed as a divergence-quasiconcavity property. Let us explain the link between quasiconcavity and upper semi-continuity of the related functional by considering the dual of these objects, namely quasiconvexity and lower-semicontinuity, that have received much more attention in the literature. We will use as a domain the $n$-dimensional torus $\T^n$ simply because it is the domain we will use throughout the paper, but more generally one could consider any $\Omega \subset \R^n$ with $|\partial \Omega| = 0$. The general question one poses is the following: given a continuous integrand $f:\R^N \to \R$ with growth
\begin{equation}\label{growth}
|f(z)| \le C(1 + \|z\|^p),
\end{equation}
under which conditions is the functional
\[
\E(z) \doteq \int_{\T^n}f(z(x))\dx 
\]
defined, for instance, for $z \in L^q(\T^n,\R^N)$, $ q \le p$, sequentially weakly lower semi-continuous? The first example of such problem was studied by C.B. Morrey in the case in which $N = m\times n$ $z(x) = \nabla u(x)$, where $u: \T^n \to \R^m$ is a $W^{1,q}$ function. In \cite{MORB}, he introduced the notion of \textit{quasiconvexity}, that is:
\begin{equation}\label{quasiconv}
f(A) \le \int_{\T^n}f(A + \nabla \phi(x))\dx, \qquad \forall \phi \in C^\infty(\T^n,\R^m), \forall A \in \R^{m\times n}.
\end{equation}
It can be proved that \eqref{growth} and \eqref{quasiconv} imply the weak lower semi-continuity of the functional $\E(\cdot)$, when $q < p$. More generally, one is interested, as we do in the present paper, in maps $z: \T^n \to \R^N$ satisfying more general constraints than $z(x) = \nabla u(x)$. The general framework, considered for instance in \cite{FM,FM1}, consists in taking a differential operator of order $k$ with smooth coefficients, usually denoted by $\mathscr{A}$, of the form
\[
\mathscr{A} \doteq \sum_{|\alpha|\le k} A_{\alpha}\partial_\alpha,\quad  \; A_\alpha \in C^\infty(\T^n,\R^{\ell \times N}).
\]
In \cite{FM} it is proved that $f$ is weakly lower-semicontinuous on $L^q(\T^n,\R^N)\cap\ker(\mathscr{A})$, $q < p$, provided that $\mathscr{A}$ satisfies Murat's constant rank condition (see \cite{FM} or \cite{MURCOM} for the definition), $f$ satisfies $\eqref{growth}$ and is $\mathscr{A}$-quasiconvex, in the sense that
\[
f(A) \le \int_{\T^n}f(A + z(x))\dx, \qquad \forall A \in \R^{N}, \forall z \in C^\infty(\T^n,\R^m) \text{ with } \mathscr{A}z = 0.
\]
The main ingredients of the proof of \cite{FM} are suitable projections of functions $z \in L^q(\T^n,\R^N)$ onto $\ker(\mathscr{A})$, and, similarly to the classical work \cite{KIP}, homogenization for Young measure (that will be introduced later on in the paper). In the last years, also due to the introduction of new techniques and concepts in the theory of Young Measures, see \cite{AB,KR,DM}, and a better understanding of the singular part of measures $\mu \in \mathcal{M}(\T^n,\R^N)$ with $\mathscr{A}\mu = 0$, see \cite{GUIANN,DIM}, there has been much progress in the study of lower-semicontinuity or relaxation of functionals, see \cite{KR,REL,GR} and the references therein.
\\
\\
As said, our paper studies upper-semicontinuity properties of the functional $\D(\cdot)$. We define a topology on $X_p$ by saying that, if $A_k,A \in X_p$, $A_k\rightharpoonup A$ in $X_p$ if $A_k \rightharpoonup A$ in $L^p$ ($A_k \overset{*}{\rightharpoonup} A$ if $p = \infty$) and $\dv A_k \overset{*}{\rightharpoonup} \dv A$ in $\mathcal{M}(\T^n,\R^n)$. The main result, contained in Section \ref{mainS}, is the following
\begin{NTeo}
Let $p > \frac{n}{n - 1}$ and $\{A_k\}_k\subset X_p$ be such that $A_k \rightharpoonup A$ in $X_p$. Then we have
$$
\limsup_k \D (A_k)\leq \D(A).
$$
\end{NTeo}
The method used to prove this result differs from the one of \cite{FM}, in that we do not use projections on the Fourier coefficients, but instead we use an homogenization argument combined with the strategy developed in \cite{SER}. The main difficulty in applying the techniques of \cite{FM} stems from the fact that our objects have image in a convex subset, $\Sym^+(n)$, of a vector space, hence the resulting projectors would not be linear. In Section \ref{counterS}, we show the optimality of the assumption $p > \frac{n}{n - 1}$ in the following

\begin{NProp}
For every $\varepsilon > 0$ and for every $x_0 \in \R^n$, there exists a sequence of matrix fields $A_k$ such that
\begin{enumerate}[(i)]
\item $A_k$ is compactly supported in $B_\varepsilon(x_0)$ for every $  k \in \N$; \label{support}
\item $A_k \rightharpoonup 0$ in $L^{\frac{n}{n - 1}}(\R^n,\Sym^+(n))$ and strongly in $L^p(\R^n,\Sym^+(n))$, $\forall p < \frac{n}{n - 1}$; \label{convergence}
\item $\dv(A_k) \in \mathcal{M}(\T^n,\R^n)$ for every $k$ and $\sup_{k \in \N}\|\dv(A_k)\|_{\mathcal{M}(\T^n,\R^n)} \le 1$; \label{divergence}
\item $\D(A_k) = \omega_n$ for every $ k$, so that in particular $\D(0) = 0 < \limsup_k\D(A_k)$.\label{determinant}
\end{enumerate}
\end{NProp}

We now introduce the notation and we state some useful and well-known results.

\subsection{Notation and technical preliminaries}

We will denote with $\T^n$ the $n$-dimensional torus of $\R^n$, that is defined as $\R^n/\Z^n$. We identify $\T^n$ with $[0,1]^n$, so that $|\T^n| = 1$, where $|E|$ denotes the $n$-dimensional Lebesgue measure of the Borel set $E\subset \R^n$. Moreover, we see every function $f: \T^n \to \R^m$ as a $\Z^n$-periodic function defined on $\R^n$, i.e. $f(x+ z) = f(x),\forall x \in \R^n, z \in \Z^n$. We denote by $\mathcal{M}(\T^n,\R^m)$ the space of bounded Radon measures with values in $\R^m$. When $m =1$, we denote this space by $\mathcal{M}(\T^n)$, and the space of positive Radon measures by $\mathcal{M}_+(\T^n)$. We recall that this is a normed space, where the norm is given by $$\|\mu\|_{\mathcal{M}(\T^n,\R^m)} \doteq \sup_{\Phi \in C^0(\T^n,\R^m),\|\Phi\|_{\infty}\le 1}\mu(\Phi),$$
and the weak-star convergence of $\mu_k \in \mathcal{M}(\T^n,\R^m)$ to $\mu \in \mathcal{M}(\T^n,\R^m)$ is given by
\[
\mu_k \weak \mu  \Leftrightarrow \mu_k(\Phi) \to \mu(\Phi), \, \, \forall \Phi \in C^0(\T^n,\R^m).
\]
Since $\mathcal{M}(\T^n,\R^m)$ is the dual of $C^0(\T^n,\R^m)$ that is a separable space, we have sequential weak-$*$ compactness for equibounded sequences $\mu_k \in \mathcal{M}(\T^n,\R^m)$ (see \cite[Section 1.9]{EVG}).\\
For every $\mu \in \mathcal{M}(\T^n,\R^m)$, we consider its Lebesgue decomposition
\[
\mu = g\dx + \mu^s,
\]
where $g \in L^1(\T^n,\R^m)$ and $\mu^s \in \mathcal{M}(\T^n,\R^m)$ denotes a singular measure with respect to the Lebesgue measure, i.e. there exists a set $A\subset \T^n$ such that $|A| = 0$ and $$\mu^s(E) = \mu^s(A\cap E), \quad \text{for every Borel set } E \subset \T^n.$$
We recall that a Lebesgue point for a function $g \in L^1(\T^n,\R^m)$ is a point $x$ such that
\[
\fint_{B_r(x)}\left|g(y) - g(x)\right|dy \to 0\, \,  \text{ as } r \to 0^+,
\]
where $$\fint_{E} f(y)\dy = \frac{1}{|E|}\int_E f(y)\dy,$$ for every $f \in L^1(\R^n)$, $ E$ Borel subset of $\R^n$  with $|E|>0$. It is well know that the set of Lebesgue points of such a function $g$ are of full measure in $\R^n$ (see \cite[Theorem 1.33]{EVG}). More generally, if $\mu \in \mathcal{M}_+(\T^n)$ or $\mathcal{M}_+(\R^n)$, we call its (upper) density the function
\[
D\mu(x) \doteq \limsup_{r \to 0^+}\frac{\mu(B_r(x))}{\omega_nr^n},
\]
where $\omega_n \doteq |B_1(0)|.$ We will use the fact that, if $\mu$ is singular with respect to the Lebesgue measure, then $D\mu(x) = 0$ for a.e. point of $\T^n$ (see \cite[Theorem 1.31]{EVG}).\\
 For symmetric matrices $A,B \in \Sym^+(n)$, we use the standard notation
\[
A \ge B
\]
to denote the partial order relation
\[
(Av,v) \ge (Bv,v),\quad  \forall v \in \R^n.
\]
Recall the basic monotonicity property of the determinant
\[
A \ge B \Rightarrow \det(A) \ge \det(B).
\]
For a matrix $A$, we denote with $P_A(\lambda)$ its characteristic polynomial, i.e.
\[
P_A(\lambda) \doteq \det(\lambda\id - A).
\]
Let us define, for a matrix $A \in \Sym^+(n)$ with eigenvalues $\lambda_1,\dots,\lambda_n$,
\[
M_i(A) \doteq \sum_{1\le j_1\le\dots\le j_i \le n}\lambda_{j_1}\dots\lambda_{j_i}, \quad \forall i \in \{1,\dots, n\},\; M_0(A) \doteq 1.
\]
It is a basic Linear Algebra fact that, if $0 \le i \le n$ the $i$-th coefficient of $P_A(\lambda)$ is given by $(-1)^{i + n}M_{n - i}(A)$. Notice in particular that $M_n(A) = \det(A)$. In the proof of Theorem \ref{t_main_sup}, we will need the following result (see \cite[Section 3.1]{DMU}):

\begin{Teo}[Fundamental Theorem on Young measure]\label{YoungT}
Let $E\subset \R^n$ be a Lebesgue measurable set with finite measure. Consider a sequence $z_k:E\subset \R^n \to \R^N$ of measurable functions satisfying the condition
\[
\sup_{k \in \N}\int_E\|z_k\|^s< +\infty,
\]
for some $s>0$. Then there exists a subsequence $z_{k_j}$ and a weak-$*$ measurable map $\nu:E \to \mathcal{M}(\R^N)$ such that for a.e. $x \in E$, $\nu_x \in \mathcal{M}_+(\R^N)$ and in addition $\nu_x(\R^N) = 1$. Moreover, for every $A \subset E$, and for every $f \in C(\R^N)$, if
\[
f(z_{k_j}) \text{ is relatively weakly compact in } L^1(A),
\]
then,
\[
f(z_{k_j}) \rightharpoonup \bar f \text{ in } L^1(A), \text{ where } \bar f(x) =\langle \nu_x, f\rangle = \int_{\R^N}f(y)d\nu_x(y).
\]
In this case, we say that $z_{k_j}$ {generates the Young measure} $\nu$.
\end{Teo}

\section{The case $p > \frac{n}{n -1}$}\label{mainS}

In this section we prove weak upper semi-continuity of the functional $\D(\cdot)$. Fix $p \in [1,\infty]$. Consider the space
$$
X_p\doteq\left\{  A\in L^p(\T^n,\Sym^+(n)) : \dv A\in\mathcal{M}(\T^n,\R^n) \right\}.
$$
We recall that $A_k\rightharpoonup A$ in $X_p$ if $A_k \rightharpoonup A$ in $L^p$ ($A_k \overset{*}{\rightharpoonup} A$ if $p = \infty$) and $\dv A_k \overset{*}{\rightharpoonup} \dv A$ in $\mathcal{M}(\T^n,\R^n)$. We prove the following
\begin{Teo}\label{t_main_sup}
Let $p > \frac{n}{n - 1}$ and $\{A_k\}_k\subset X_p$ be such that $A_k \rightharpoonup A$ in $X_p$. Then we have
$$
\limsup_k \D(A_k)\leq \D(A).
$$
\end{Teo}

To prove Theorem \ref{t_main_sup} we follow the argument of \cite{FM}, indeed we will prove that the Young measure $\nu =(\nu_x)_{x\in \T^n}$ generated by the sequence $\{A_k\}_k$, satisfies
\begin{equation}\label{FM_ineq}
\langle \nu_x, \det(\cdot)^\frac{1}{n-1}\rangle \leq \det(A(x))^\frac{1}{n-1},
\end{equation}
for almost every $x\in \T^n$. Indeed, by the Fundamental Theorem  of Young measures and \eqref{FM_ineq}, we would conclude 
$$
\limsup_k \D(A_k)=\lim_k \D(A_k)= \int_{\T^n} \langle \nu_x, \det(\cdot)^\frac{1}{n-1}\rangle \,\dx \overset{\eqref{FM_ineq}}{\leq } \D(A),
$$
i.e. the weak upper semi-continuity of $\D(\cdot)$ on $X_p$, where in the first equality we used the fact that up to a subsequence we can further suppose that $\limsup_k \D(A_k)=\lim_k \D(A_k)$. The argument to obtain \eqref{FM_ineq} is different to the one given in \cite{FM}  and heavily relies on the ideas of \cite[Proof of Theorem 2.2]{SER}. First we make the following remarks of technical nature.

 \begin{rem}\label{pos}
 We remark that it is sufficient to prove the theorem in the case in which $A_k, A\geq \varepsilon \id_n$ for some $\varepsilon>0$. Indeed, in the general case one can consider 
 $A_k^\varepsilon=A_k+\varepsilon\id_n$, for which one proved weak upper semi-continuity of $\D$, meaning that 
 $$
 \limsup_k \D(A_k^\varepsilon)\leq \D(A^\varepsilon).
 $$
By monotonicity of the determinant on the cone of positive definite matrices, we also have
$$
\limsup_k \D(A_k)\leq  \limsup_k \D(A_k^\varepsilon)\leq \D(A^\varepsilon),
$$
thus the theorem in the general case follows by letting $\varepsilon\rightarrow 0$.
\end{rem}
\begin{rem}\label{smooth}
We can also suppose that the sequence $\{A_k\}_k$ is smooth.  Indeed for any $A_k\in X_p$ there exists a smooth matrix field $\tilde A_k\in X_p$ such that 
\begin{enumerate}[(i)]
\item $\|A_k -\tilde A_k\|_{L^p(\T^n)}\leq \frac{1}{k}$;\label{strong}
\item $\int_{\T^n} \| \dv (\tilde A_k)\|(x) \dx \leq \|\dv (A_k)\|_{\mathcal{M}(\T^n,\R^n)}$ for every $k$;\label{measure}
\item $\tilde A_k \rightharpoonup A$ in $X_p$.\label{weakcon}
\end{enumerate}
To construct it, consider a standard family of mollifiers $\rho_\varepsilon(x) = \frac{1}{\varepsilon^n}\rho\left(\frac{x}{\varepsilon}\right)$, where $$\rho \in C_c^\infty(B_1(0)),\quad  \int_{\R^n}\rho(x)\dx = 1,\quad  \rho(x) \ge 0, \forall x \in \R^n$$ and consequently $A_{k,\varepsilon}(x)\doteq A_k* \rho_\varepsilon(x)$. Clearly $A_{k,\varepsilon} \in X_p$ and is smooth $\forall k,\varepsilon$. As $\varepsilon \to 0$, we have that $A_{k,\varepsilon} \to A_k$ for fixed $k$ in $L^p(\T^n,\Sym^+(n))$. Hence, for every $k$ we can choose $\varepsilon_k$ such that $\eqref{strong}$ is fulfilled. Define $\tilde A_k \doteq A_{k,\varepsilon_k}$. We need to show \eqref{measure} and \eqref{weakcon}. Since mollification does not increase the total mass, we have
\[
\|\tilde A_k\|_{L^p} \le \|A_k\|_{L^p}, \quad  \|\dv (\tilde A_k)\|_{\mathcal{M}(\T^n,\R^n)} \le \|\dv (A_k)\|_{\mathcal{M}(\T^n,\R^n)}, \forall k \in \N.
\]
The second inequality is exactly $\eqref{measure}$. Moreover, by the weak convergence in $X_p$, both $\|A_k\|_{L^p}$ and $\|\dv A_k\|_{\mathcal{M}(\T^n,\R^n)}$ are equibounded sequences, hence $\tilde A_k$ is precompact in $X_p$, in the sense that for every subsequence, there exists a further subsequence converging in $X_p$ to some tensor field $B \in X_p$. By \eqref{strong}, any limit point of this sequence with respect to the topology of $X_p$ must be the same as the one of $A_k$, namely $A$, hence \eqref{weakcon} follows. Thus, if Theorem \ref{t_main_sup} is true for a smooth sequence, we have
\begin{equation}\label{inequal}
\begin{split}
\limsup_k \D(A_k)&=\limsup_k \left(\D(A_k)-\D(\tilde A_k)+\D(\tilde A_k)\right)\\\
&\leq \limsup_k \left(\D(A_k)-\D(\tilde A_k)\right)+\limsup_k \D(\tilde A_k)\le \D(A).
\end{split}
\end{equation}
Let us justify the last inequality. We can estimate, using the \"Holder property of $t \mapsto t^{\frac{1}{n -1}}$,
\[
|\D(A_k)-\D(\tilde A_k)| \le \int_{\T^n}|\det(A_k(x))^{\frac{1}{n - 1}} - \det(\tilde A_k(x))^{\frac{1}{n - 1}}|\dx \le\int_{\T^n}|\det(A_k(x)) - \det(\tilde A_k(x))|^{\frac{1}{n - 1}}\dx.
\]
Moreover, a simple estimate valid for every couple of matrices $X, Y \in \R^{n\times n}$ gives, for some dimensional constant $c > 0$,
\[
|\det(X) - \det(Y)| \le c(\|X\|^{n - 1} + \|Y\|^{n - 1})\|X - Y\|.
\]
Therefore, using this inequality and the subadditivity of $t \mapsto t^{\frac{1}{n -1}}$
\begin{align*}
&\int_{\T^n}|\det(A_k(x)) - \det(\tilde A_k(x))|^{\frac{1}{n - 1}}\dx \\
&\le c^{\frac{1}{n -1}}\int_{\T^n}\left(\|A_k(x)\|^{n - 1} + \|\tilde A_k(x)\|^{n - 1}\right)^{\frac{1}{n -1}}\|A_k(x) - \tilde A_k(x))\|^{\frac{1}{n - 1}}\dx\\
&\le c^{\frac{1}{n -1}}\int_{\T^n}\left(\|A_k(x)\| + \|\tilde A_k(x)\|\right)\|A_k(x) - \tilde A_k(x))\|^{\frac{1}{n - 1}}\dx\\
&\le c^{\frac{1}{n -1}}\left(\int_{\T^n}\left(\|A_k(x)\| + \|\tilde A_k(x)\|\right)^{\frac{n}{n - 1}}\dx\right)^{\frac{n - 1}{n}}\left(\int_{\T^n}\|A_k(x) - \tilde A_k(x))\|^{\frac{n}{n -1}}\dx\right)^{\frac{1}{n}},
\end{align*}
the last inequality being \"Holder inequality with exponents $\frac{n}{n - 1}$ and $n$. The previous inequality and \eqref{strong} justify the last estimate of  \eqref{inequal}.
\end{rem}

\begin{proof}[Proof of Theorem \ref{t_main_sup}]

First notice that up to (non-relabeled) subsequences we can suppose $$\limsup_k \D(A_k)=\lim_k \D(A_k)$$ and that $\{A_k\}_k$ generates the Young measure $\nu = (\nu_x)_{x\in \T^n}$. From Remark \ref{pos} and Remark \ref{smooth}, we can further suppose that both $A_k, A\geq \varepsilon \id_n$ for some $\varepsilon>0$ and $A_k$ are smooth.
\\
\\
\indent\fbox{Step 1: definition of the main objects}
\\
\\
Let $\mu_{k} \in \mathcal{M}_+(\T^n)$ be the finite Radon measures defined by $\mu_k(E) \doteq \int_E\|\dv(A_k)\|(x)\dx$ and call $\mu$ its weak-* limit (that we can always suppose to exist up to further subsequences). Notice that, for every $i \in \{1,\dots,n\}$, the map
\[
x \mapsto M^{\frac{1}{n - 1}}_i(A_k(x))
\]
is equibounded in $L^{\frac{p(n - 1)}{i}}(\T^n)$. Since $p > \frac{n}{n - 1}$ and $i \le n$, these sequences fulfill the hypotheses of Theorem \ref{YoungT}, hence
\[
M^{\frac{1}{n - 1}}_i(A_k(x)) \rightharpoonup \langle\nu_x, M_i^{\frac{1}{n -1}}(\cdot)\rangle\quad  \text{ in } L^1(\T^n).
\]
Consider $T' \subset \T^n$ to be the set of points $a \in \T^n$ such that
\begin{itemize}
\item $\|A(a)\|<\infty$;
\item $\langle \nu_a, M_i^{\frac{1}{n-1}}(\cdot)\rangle < + \infty, \forall i \in \{0,\dots,n\}$;
\item $a$ is a Lebesgue point for $x \mapsto A(x)$;
\item $a$ is a Lebesgue point for $x \mapsto \langle \nu_x, M_i^{\frac{1}{n-1}}(\cdot)\rangle$, for $i \in \{1,\dots,n\}$.
\end{itemize}
Since these are  $L^1(\T^n)$ functions, we get $|\T^n\setminus T'| = 0$.  Let $\mu = g\dx + \mu^s$ be the Lebesgue decomposition of the weak-* limit of $\mu_k$, and define $T'' \subset \T^n$ to be the set of points that are both Lebesgue points for $g$ and density $0$ points for $\mu^s$. By \cite[Theorem 1.31]{EVG}, $|\T^n\setminus T''| =0$. Finally, define $T \doteq T'\cap T''\cap (0,1)^n$. As explained before the proof of the theorem, we want to prove \eqref{FM_ineq}, namely
\begin{equation*}
\langle \nu_a,\det(\cdot)^{\frac{1}{n -1}}\rangle \le \det(A(a))^{\frac{1}{n - 1}}, \quad \forall a \in T.
\end{equation*}
Therefore, from now on we fix $a \in T$. Consider a cut-off function  $\varphi \in C^\infty_c((0,1)^n)$, $0\leq \varphi\leq 1$. For $k \in \N$ and $R>0$, we define $B_{k,R}$ over $(0,1)^n$ by
$$
B_{k,R}(x)\doteq \varphi(x) A_{k}(a+Rx) + (1 - \varphi(x))A(a).
$$
Remark that $B_{k,R}\equiv A(a)$ over the boundary of $[0,1]^n$, therefore $B_{k,R}$ can be extended smoothly by periodicity to $\R^n$. This defines $B_{k,R}$ over $\T^n$. Notice moreover that $B_{k,R}$ takes values in $\Sym^+(n)$.
\\
\\
\indent\fbox{Step 2: Monge-Amp\`ere and the main inequality}
\\
\\
The argument of this step is the same as the one of \cite[Theorem 2.2]{SER}. Let $\phi_{k,R}: \T^n \to \R$ be the solution of 
\begin{equation}\label{eq_c}
\det(H\phi_{k,R} + S_{k,R}) = \det(B_{k,R})^{\frac{1}{n - 1}}\doteq f_{k,R},
\end{equation}
where $H\phi_{k,R}(x) + S_{k,R}(x) \in \Sym^+(n), \forall x \in \T^n$, with the constraint 
\begin{equation}\label{con}
\det(S_{k,R}) = \int_{\T^n}f_{k,R}(x)\, \dx.
\end{equation}
From \cite[Theorem 2.2]{YAN}, it is known that the latter is a necessary and sufficient condition to solve the Monge Amp\`ere type equation \eqref{eq_c}.
Note that \eqref{eq_c} is equivalent to 
\begin{equation}\label{eq_c_2}
\det(H\psi_{k,R} ) =  f_{k,R},
\end{equation}
where $H\psi_{k,R}(x)$ is positive definite $\forall x \in \T^n$ and $\psi_{k,R}(x)\doteq \frac{1}{2}x^T S_{k,R} x +\phi_{k,R}(x)$. We can, and will, assume that
\begin{equation}\label{zer}
\psi_{k,R}(a) = \phi_{k,R}(a) = 0,\quad  \forall k,R,
\end{equation}
since the solution of \eqref{eq_c_2} is determined up to constants (see again \cite[Theorem 2.2]{YAN}). We have 
$$
f_{k,R}=\left(f_{k,R}\det(B_{k,R})\right)^\frac{1}{n}=\left( \det(H\psi_{k,R} B_{k,R} ) \right)^\frac{1}{n}.
$$
Since, for every $x \in \T^n$, $k \in \N$, $R > 0$, $H\psi_{k,R}(x)B_{k,R}(x)$ is the product of two symmetric and positive definite matrices, their product is diagonalizable with positive eigenvalues (see \cite[Proposition 6.1]{SERBOOK}). Dropping the dependence of $k,R,x$, if we call these eigenvalues $\lambda_1,\dots, \lambda_n$ we can write
\[
f_{k,R} = \left( \det(H\psi_{k,R} B_{k,R} ) \right)^\frac{1}{n} = (\lambda_1\dots\lambda_n)^{\frac{1}{n}} \le \frac{\sum_{i = 1}^n\lambda_i}{n},
\]
where in the last inequality we use the arithmetic-geometric mean inequality. Hence,
\[
f_{k,R} \le \frac{\tr(H\psi_{k,R} B_{k,R})}{n}.
\]
Using the definition of $\psi_{k,R}$ and rewriting
\[
\tr(H\phi_{k,R}B_{k,R}) = \dv(B_{k,R}\nabla \phi_{k,R}) - (\dv(B_{k,R}),\nabla\phi_{k,R}),
\]
we finally get 
\begin{equation}\label{tbin}
f_{k,R}\le\frac{1}{n} (\tr(B_{k,R} S_{k,R} )+\dv(B_{k,R}\nabla \phi_{k,R}) - (\dv(B_{k,R}),\nabla\phi_{k,R})).
\end{equation}
We consider $S_{k,R}$ of the form
\[
S_{k,R} = \lambda_{k,R}\cof\left(\int_{\T^n}B_{k,R}(x)\dx\right).
\]
By \eqref{con}

\begin{equation}\label{lam}
\displaystyle\lambda_{k,R}=\frac{\left(\int_{\T^n}\det(B_{k,R})^{\frac{1}{ n - 1}}(x)\dx\right)^{\frac{1}{n}}}{\left(\det(\int_{\T^n}B_{k,R}(x)\dx)\right)^{\frac{n - 1}{n}}}.
\end{equation}

Observing that $\int_{\T^n}\dv(B_{k,R}\nabla \phi_{k,R})\dx = 0$, we integrate \eqref{tbin}, getting
\begin{equation}\label{main}
\int_{\T^n} \det(B_{k,R})^\frac{1}{n - 1}\dx \le \frac{1}{n}\int_{\T^n}\tr(B_{k,R}S_{k,R})\dx- \frac{1}{n}\int_{\T^n}(\dv(B_{k,R}),\nabla\phi_{k,R}))\dx.
\end{equation}
We rewrite
\begin{align*}
\int_{\T^n}\tr(B_{k,R}S_{k,R})\dx &= \tr\left(\left(\int_{\T^n}B_{k,R}\dx\right) S_{k,R}\right) = \lambda_{k,R}\tr\left(\left(\int_{\T^n}B_{k,R}\dx\right) \cof\left(\int_{\T^n}B_{k,R}(x)\dx\right)\right)\\
& = n\lambda_{k,R}\det\left(\int_{\T^n}B_{k,R}(x)\dx\right) \\
& \overset{\eqref{lam}}{=} n\left(\int_{\T^n}\det(B_{k,R})^{\frac{1}{ n - 1}}(x)\dx\right)^{\frac{1}{n}}\left(\det\left(\int_{\T^n}B_{k,R}(x)\dx\right)\right)^{\frac{1}{n}}.
\end{align*}

Finally, define also $\gamma_{k,R} \doteq \left(\int_{\T^n}\det(B_{k,R})^{\frac{1}{ n - 1}}(x)\dx\right)^{\frac{1}{n}}$. By the monotonicity of the determinant and the fact that $A_k(x) \ge \varepsilon \id_n,\forall x \in \T^n, \forall k \in \N$, and $A(a) \ge \varepsilon\id_n$, we have $B_{k,R} \ge \varepsilon\id_n, \forall k,R$, that implies
\begin{equation}\label{bound}
\gamma_{k,R} \ge  \varepsilon^{\frac{1}{n - 1}},\quad  \forall k,R.
\end{equation}
We divide by $\gamma_{k,R}$ in \eqref{main}, to obtain:

\begin{equation}\label{equaz1}
\left(\int_{\T^n}\det(B_{k,R})^{\frac{1}{ n - 1}}(x)\dx\right)^{\frac{n - 1}{n}} \le \left(\det\left(\int_{\T^n}B_{k,R}(x)\dx\right)\right)^{\frac{1}{n}}- \frac{1}{n\gamma_{k,R}}\int_{\T^n}(\dv(B_{k,R}),\nabla\phi_{k,R}))\dx
\end{equation}
By monotonicity of the determinant we have
$$
\int_{\T^n} \varphi(x)^{\frac{n}{n-1}}\det(A_{k}(a+Rx))^\frac{1}{n - 1}\dx \le \int_{\T^n}\det(B_{k,R})^{\frac{1}{ n - 1}}(x)\dx,
$$
so that \eqref{equaz1} becomes
\begin{equation}\label{eq}
\begin{split}
&\left(\int_{\T^n} \varphi(x)^{\frac{n}{n-1}}\det(A_{k}(a+Rx))^\frac{1}{n - 1}\dx\right)^{\frac{n - 1}{n}}\\
& \le \left(\det\left(\int_{\T^n}B_{k,R}(x)\dx\right)\right)^{\frac{1}{n}}- \frac{1}{n\gamma_{k,R}}\int_{\T^n}(\dv(B_{k,R}),\nabla\phi_{k,R}))\dx
\end{split}
\end{equation}
thus by denoting 
\begin{equation}\nonumber
I_{k,R} \doteq \int_{\T^n} \varphi(x)^{\frac{n}{n-1}}\det(A_{k}(a+Rx))^\frac{1}{n - 1}\dx,
\end{equation}

\begin{equation}\nonumber
II_{k,R} \doteq \det\left(\int_{\T^n}B_{k,R}(x)\dx\right),
\end{equation}

\begin{equation}\nonumber
III_{k,R}\doteq \int_{\T^n}(\dv(B_{k,R}),\nabla\phi_{k,R}))\dx,
\end{equation}
we can put \eqref{eq} in a more compact form:
\begin{equation}\label{equaz}
I_{k,R}^\frac{n-1}{n}\leq II^\frac{1}{n}_{k,R}- \frac{1}{n\gamma_{k,R}}III_{k,R}.
\end{equation}
We will first let $k \to + \infty$ and then $R \to 0$. To this aim, we study separately the three terms.
\\
\\
\indent\fbox{Step 3: $I_{k,R}$}
\\
\\
Denoting $Q_R=a+ [0,R]^n$ we have
$$
I_{k,R}=\int_{Q_R} \varphi^{\frac{n}{n-1}}\left(\frac{y-a}{R}\right)\det(A_{k}(y))^\frac{1}{n - 1}\frac{\dy}{R^n}.
$$
Since the sequence $A_{k}$ generates the Young measure $\nu$, by letting $k\to \infty$, we get 
$$
\lim_{k \to \infty} I_{k,R}= \int_{Q_R} \varphi^{\frac{n}{n-1}}\left(\frac{y-a}{R}\right) \langle \nu_y, \det(\cdot) ^\frac{1}{n-1}\rangle \frac{\dy}{R^n}= \int_{\T^n} \varphi^{\frac{n}{n-1}}\left(x\right) \langle \nu_{a+Rx}, \det(\cdot) ^\frac{1}{n-1}\rangle\dx.
$$
Finally, since  $a\in (0,1)^n$ was a Lebesgue point for the function $x \mapsto\langle \nu_{x}, \det(\cdot) ^\frac{1}{n-1}\rangle$, letting $R\to 0$ we achieve 
$$
\lim_{R\to 0 } \lim_{k\to \infty}I_{k,R} = \langle \nu_a, \det (\cdot)^\frac{1}{n-1}\rangle \int_{\T^n}  \varphi^{\frac{n}{n-1}}(x) \dx.
$$
\\
\\
\indent\fbox{Step 4: $II_{k,R}$}
\\
\\
Since $A_k \rightharpoonup A$ in $L^p(\T^n)$, we have 
\begin{equation}\label{est}
\lim_{k\to \infty} II_{k,R} = \det\left( \int_{\T^n} \varphi(x) A(a+Rx) \dx + A(a) \int_{\T^n} 1-\varphi(x) \dx\right),
\end{equation}
and since 
\[
 \int_{\T^n} \varphi(x) A(a+Rx) \dx=\int_{Q_R} \varphi\left(\frac{y-a}{R}\right)A(y) \frac{\dy}{R^n}
\]
and $|\varphi(x)|\le 1\; \forall x \in \T^n$, we also get that 
\[
\left\|\int_{\T^n} \varphi(x) A(a+Rx) \dx-A(a) \int_{\T^n} \varphi(x) \dx \right\|\leq \int_{Q_R} \|A(y)-A(a)\|\frac{\dy}{R^n}.
\]
The last expression tends to $0$ as $R \to 0^+$, since $a$ is a Lebesgue point for $x \mapsto A(x)$. Thus, by letting $R\to 0$ in \eqref{est}, we conclude that
$$
\lim_{R\to 0} \lim_{k\to \infty} II_{k,R} = \det(A(a)).
$$
\\
\indent\fbox{Step 5: $III_{k,R}$}
\\
\\
To prove \eqref{FM_ineq}, we are just left to show that $\lim_{R\to 0}\lim_{k\to \infty} III_{k,R}=0$. To do this, we first compute $$\dv(B_{k,R}) = \varphi(x)R\dv(A_k)(a + Rx) + (A_k(a + Rx) -A(a))\nabla\varphi(x).$$ Therefore:
\begin{align*}
III_{k,R} &= R\int_{\T^n}\varphi(x)(\dv(A_k)(a + Rx),\nabla\phi_{k,R})\dx \\
&+ \int_{\T^n}((A_k(a + Rx) -A(a))\nabla\varphi,\nabla\phi_{k,R})\dx.
\end{align*}
We can use the divergence theorem to rewrite more conveniently the second term:
\begin{align*}
\int_{\T^n}((A_k(a + Rx) -A(a))\nabla\varphi,\nabla\phi_{k,R})\dx = \\
\sum_{i,j}\int_{\T^n}((A_k)_{ij}(a + Rx) -A_{ij}(a))\partial_j\varphi\partial_i\phi_{k,R}\dx = \\
-\sum_{i,j}\int_{\T^n}\partial_i((A_k)_{ij}(a + Rx) -A_{ij}(a))\partial_j\varphi\phi_{k,R}\dx  \\
-\sum_{i,j}\int_{\T^n}((A_k)_{ij}(a + Rx) -A_{ij}(a))\partial_{ij}\varphi \phi_{k,R}\dx = \\
-R\sum_{i,j}\int_{\T^n}(\partial_iA_k)_{ij}(a + Rx) \partial_j\varphi\phi_{k,R}\dx  \\
-\sum_{i,j}\int_{\T^n}((A_k)_{ij}(a + Rx) -A_{ij}(a))\partial_{ij}\varphi \phi_{k,R}\dx = \\
-R\int_{\T^n}((\dv A_k)(a + Rx), \nabla\varphi)\phi_{k,R}\dx  \\
-\int_{\T^n}(A_k(a + Rx) -A(a)),H\varphi) \phi_{k,R}\dx.
\end{align*}
Summarizing, we have
\begin{align*}
III_{k,R} &= R\int_{\T^n}\varphi(x)(\dv(A_k)(a + Rx),\nabla\phi_{k,R})\dx \\
&-R\int_{\T^n}((\dv A_k)(a + Rx), \nabla\varphi)\phi_{k,R}\dx  \\
&-\int_{\T^n}(A_k(a + Rx) -A(a)),H\varphi) \phi_{k,R}\dx.
\end{align*}
We will denote with:
\begin{equation}\nonumber
III^1_{k,R} \doteq R\int_{\T^n}\varphi(x)(\dv(A_k)(a + Rx),\nabla\phi_{k,R})\dx ,
\end{equation}

\begin{equation}\nonumber
III^2_{k,R} \doteq R\int_{\T^n}((\dv A_k)(a + Rx), \nabla\varphi)\phi_{k,R}\dx ,
\end{equation}

\begin{equation}\nonumber
III^3_{k,R} \doteq \int_{\T^n}(A_k(a + Rx) -A(a)),H\varphi) \phi_{k,R}\dx.
\end{equation}
\\
\\
\indent\fbox{Step 6: Estimates on $\phi_{k,R}$}
\\
\\
As remarked in \cite[Section 5.2]{SER}, $\psi_{k,R}$ is convex, for every $k,R$, and moreover the estimate
\begin{equation}\label{comp}
\|\nabla \phi_{k,R}\|_{L^\infty(\T^n)} \le C\|S_{k,R}\|
\end{equation}
holds for every $k \in \N$ and $R >0$. We will now show that
\begin{equation}\label{imp}
\limsup_{R\to 0^+}\limsup_{k \to +\infty}\|S_{k,R}\| < + \infty.
\end{equation}
If we do this, we find, through a diagonal argument, a subsequence $k_j$ such that  $\phi_{k_j,\frac{1}{m}}$ converges uniformly to a function $\phi_{\frac{1}{m}}$ as $j \to \infty$. Moreover we find a constant $\lambda > 0$ such that
\begin{equation}\label{unif}
\|\phi_{\frac{1}{m}}\|_{C^0(\T^n)} \le \lambda, \quad \forall m \in \N.
\end{equation}
Let us first show how \eqref{imp} implies this last claim. By \eqref{zer} we have $\phi_{k,\frac1m}(a) = 0, \forall k,m$, and the estimate \eqref{comp} combined with $\eqref{imp}$ tells us that for every fixed $m$, $\{\phi_{k,\frac1m}\}_{k \in \N}$ is a precompact subset of $C^0(\T^n)$, hence there exists a diagonal subsequence $\phi_{k_j,\frac1m}$ that converges uniformly to $\phi_{\frac1m}$ for every $m$ as $j \to \infty$. Moreover, estimate \eqref{comp} implies that for some universal constant $\alpha > 0$ 
\[
\|\phi_{k_j,\frac{1}{m}}\|_{C^0(\T^n)} \le \alpha\|S_{k_j,\frac{1}{m}}\|,\quad \forall j, m.
\]
Therefore, in the limit as $j \to \infty$, we also infer
\[
\|\phi_{\frac{1}{m}}\|_{C^0(\T^n)} \le \alpha\limsup_{k \to \infty}\|S_{k,\frac{1}{m}}\|,\quad \forall m
\]
and finally
\[
\limsup_{m \to \infty}\|\phi_{\frac{1}{m}}\|_{C^0(\T^n)} \le \alpha\limsup_{m \to \infty}\limsup_{k \to \infty}\|S_{k,\frac{1}{m}}\| \overset{\eqref{imp}}{<} +\infty,
\]
which finally implies \eqref{unif}. Let us prove \eqref{imp}. By its definition, we have $$S_{k,R} = \lambda_{k,R}\cof\left(\int_{\T^n}B_{k,R}(x)\dx\right).$$ Therefore it suffices to prove separately that
\begin{equation}\label{imp1}
\limsup_{R\to 0}\limsup_{k\to \infty}\left\|\cof\left(\int_{\T^n}B_{k,R}(x)\dx\right)\right\| < +\infty
\end{equation}
and
\begin{equation}\label{imp2}
\limsup_{R\to 0}\limsup_{k\to \infty}\lambda_{k,R} < +\infty.
\end{equation}

We start with $\eqref{imp1}$. The weak convergence of $A_k$ implies, as in \eqref{est} and the subsequent computations, that
\[
\lim_{R\to 0}\lim_{k \to \infty}\int_{\T^n}B_{k,R}(x)\dx = A(a),
\]
since $a \in T'$.
Hence
\[
\limsup_{R\to 0}\limsup_{k\to \infty}\left\|\cof\left(\int_{\T^n}B_{k,R}(x)\dx\right)\right\| = \lim_{R\to 0}\lim_{k\to \infty}\left\|\cof\left(\int_{\T^n}B_{k,R}(x)\dx\right)\right\| =  \|\cof(A(a))\| < + \infty,
\]
where the last inequality is again justified by $a \in T'$. Finally, we compute \eqref{imp2}. By definition
\[
\displaystyle\lambda_{k,R}=\frac{\left(\int_{\T^n}\det(B_{k,R})^{\frac{1}{ n - 1}}(x)\dx\right)^{\frac{1}{n}}}{\left(\det(\int_{\T^n}B_{k,R}(x)\dx)\right)^{\frac{n - 1}{n}}}.
\]
Analogously to the estimate of $\gamma_{k,R}$ of \eqref{bound}, we have
\[
\left(\det\left(\int_{\T^n}B_{k,R}(x)\dx\right)\right)^{\frac{n - 1}{n}} \ge \varepsilon^{n - 1}.
\]
Therefore, to conclude the proof, we just need to show that
\[
\limsup_{R \to 0} \limsup_{k \to \infty} \int_{\T^n}\det(B_{k,R})^{\frac{1}{ n - 1}}(x)\dx < +\infty.
\]
First note that
\[
A(a) \le \|A(a)\|\id_n,
\]
and consequently estimate
\[
\det(B_{k,R}) \le \det(\varphi(x)A_k(a+Rx) + (1 -\varphi(x))\|A(a)\|\id) = P_{-\varphi(x) A_k(a + Rx)}((1 - \varphi(x))\|A(a)\|),
\]
where $P_{-\varphi(x) A_k(a + Rx)}$ is the characteristic polynomial of $-\varphi(x) A_k(a + Rx)$. By the structure of the characteristic polynomial and the subadditivity of the function $t \mapsto t^{\frac{1}{n - 1}}$, we can bound
\begin{align*}
\det(B_{k,R})^\frac{1}{n - 1}(x) &\le |P|^{\frac{1}{n - 1}}_{-\varphi(x) A_k(a + Rx)}((1 - \varphi(x))\|A(a)\|) \\
&\le \sum_{i = 0}^n\left[(1 - \varphi(x))^i\|A(a)\|^iM_{n - i}(\varphi(x) A_k(a + Rx))\right]^{\frac{1}{n - 1}}.
\end{align*}
Since $M_{n - i}$ is $n - i$ homogeneous, $M_{n - i}(\varphi(x) A_k(a + Rx)) = \varphi^{n - i}(x)M_{n - i}(A_k(a + Rx))$. Hence
\[
\det(B_{k,R})^\frac{1}{n - 1}(x) \le \sum_{i = 0}^n\left[(1 - \varphi(x))^i\|A(a)\|^i\varphi^{n - i}(x)M_{n - i}(A_k(a + Rx))\right]^{\frac{1}{ n -1}}.
\]
Now observe that, for every $i\in\{0,1,\dots, n\}$,
\[
\int_{\T^n}\left[(1 - \varphi)^i\varphi^{n - i}M_{n - i}(A_k(a + Rx))\right]^{\frac{1}{ n -1}}\dx \to \int_{\T^n}\left[(1 - \varphi)^i\varphi^{n - i}\right]^{\frac{1}{n - 1}}\langle\nu_{a + Rx},M_{n-i}^{\frac{1}{n -1}}(\cdot)\rangle\dx
\]
as $k \to \infty$, by the Fundamental Theorem of Young measures. Letting $R \to 0^+$, since $a$ is a Lebesgue point for $x \mapsto\langle\nu_{x},M_{n-i}^{\frac{1}{n -1}}(\cdot)\rangle$, we find that
\begin{align*}
\limsup_{R\to 0^+}\limsup_{k \to \infty}\int_{\T^n}\det(B_{k,R})^{\frac{1}{n - 1}}(x)\dx &\leq \sum_{i = 0}^n\langle\nu_{a},M_{n-i}^{\frac{1}{n -1}}(\cdot)\rangle\int_{\T^n}\left[(1 - \varphi(x))^i\|A(a)\|^i\varphi^{n - i}(x)\right]^{\frac{1}{ n -1}}\dx\\
& \le \sum_{i = 0}^n\langle\nu_{a},M_{n-i}^{\frac{1}{n -1}}(\cdot)\rangle\|A(a)\|^i,
\end{align*}
the last inequality being true since $0 \le\varphi(x) \le 1, \forall x \in \T^n$. Clearly the last term is equibounded by our choice $a\in T'$. We are now going to prove that the three terms of $III_{k_j,\frac{1}{m}}$ converge to $0$ as $j \to \infty$ and $m\to \infty$.
\\
\\
\indent\fbox{Step 7: $III^1_{k_j,\frac 1m}$ and $III^2_{k_j,\frac 1m}$}
\\
\\
By \eqref{comp}, we know that $\|\nabla\phi_{k_j,\frac{1}{m}}\|_{L^\infty(\T^n)} \le C\|S_{k_j,\frac1m}\|$. Hence
\begin{align*}
|III^1_{k_j,\frac 1m}| &= \left|\frac{1}{m}\int_{\T^n}\varphi(x)(\dv(A_{k_j})\left(a + \frac{x}{m}\right),\nabla \phi_{k_j,\frac{1}{m}})\dx\right|  \\
&\leq \frac{\|\nabla\phi_{k_j,\frac 1m}\|}{m}\fint_{Q_{\frac 1m}(a)}\varphi\left(m(x - a)\right)\|\dv(A_{k_j})\|(x)dx\\
& \le \frac{C\|S_{k_j,\frac1m}\|}{m}\fint_{Q_{\frac 1m}(a)} \|\dv(A_{k_j})\|(x)dx .
\end{align*}
Recall that we use the notation $\mu_{k}(E) = \int_{E}\|\dv(A_k)\|(x)\dx$, for every Borel set $E \subset \T^n$ and for every $k \in \N$. By weak-* convergence of measures, since $Q_{\frac 1m}(a)$ is a compact set, we have (see \cite[Theorem 1.40]{EVG})
\begin{align*}
\limsup_{j \to \infty}\frac{C}{m}\frac{\mu_{k_j}(Q_{\frac1m}(a))}{(\frac1m)^n} &\le \frac{C}{m}\frac{\mu(Q_{\frac1m}(a))}{(\frac1m)^n} \le \frac{C'}{m}\frac{\mu(B_{\sqrt{2}/m}(a))}{|B_{\sqrt{2}/m}(a)|} \\
&= \frac{C'}{m}\fint_{B_{\sqrt{2}/m}(a))}g(x)\dx + \frac{C'}{m}\frac{\mu^s(B_{\sqrt{2}/m}(a))}{|B_{\sqrt{2}/m}(a)|},
\end{align*}
for some positive constant $C'$. Since we chose $a \in T''$, we get that the previous expression converges to 0 as $m \to \infty$. Finally, by \eqref{imp}, we also know that $$\limsup_{R \to 0^+}\limsup_{j \to \infty}\|S_{k_j,R}\| < + \infty,$$ hence $\limsup_{m \to \infty}\limsup_{j \to \infty}III^1_{k_j,\frac1m} = 0$. The term $III^2_{k_j,\frac1m}$ is completely analogous. 
\\
\\
\indent\fbox{Step 8: $III^3_{k_j,\frac 1m}$}
\\
\\
We finally prove that $\lim_{m \to \infty}\lim_{j \to \infty}III^3_{k_j,\frac1m} = 0$. We have
\begin{align*}
III^3_{k_j,\frac1m} &=\int_{\T^n}\left(A_{k_j}\left(a + \frac{x}{m}\right) -A(a)),H\varphi\right) \phi_{k_j,\frac1m}\dx\\
& =\int_{\T^n}\left(A_{k_j}\left(a + \frac{x}{m}\right) -A(a),H\varphi\right) (\phi_{k_j,\frac1m}- \phi_{\frac1m})\dx\\
&+ \int_{\T^n}\left(A_{k_j}\left(a + \frac{x}{m}\right) -A(a),H\varphi\right) \phi_{\frac1m}\dx.
\end{align*}
The first term can be estimated as
\begin{align*}
\left| \int_{\T^n} \left(A_{k_j}\left(a + \frac{x}{m}\right) -A(a),H\varphi\right) (\phi_{k_j,\frac1m}- \phi_{\frac1m})\dx\right| \\
\leq \|\phi_{k_j,\frac1m}- \phi_{\frac1m}\|_{C^0(\T^n)}\|H\varphi\|_{C^0(\T^n)}\int_{\T^n}\left\|A_{k_j}\left(a + \frac{x}{m}\right) - A(a)\right\|\dx \\
=\|\phi_{k_j,\frac1m}- \phi_{\frac1m}\|_{C^0(\T^n)}\|H\varphi\|_{C^0(\T^n)}m^n\int_{Q_{\frac1m}(a)}\left\|A_{k_j}(x) - A(a)\right\|\dx.
\end{align*}
Since $h_j(x) \doteq \|A_{k_j}(x) - A(a)\|$ is equibounded in $L^p(Q_{\frac1m}(a))$ and by the uniform convergence of $\phi_{k_j,\frac1m}$ to $\phi_{\frac1m}$, we infer that the last term converges to $0$ as $j \to \infty$. On the other hand, by weak $L^p$ convergence,
\[
\int_{\T^n}\left(A_{k_j}\left(a + \frac{x}{m}\right) -A(a),H\varphi\right) \phi_{\frac1m}\dx \to \int_{\T^n}\left(A\left(a + \frac{x}{m}\right) -A(a),H\varphi\right) \phi_{\frac1m}\dx
\]
as $j \to \infty$. Now, since $a$ is a Lebesgue point for $A$, and we can estimate for some constant $\gamma > 0$
\[
\left|\int_{\T^n}\left(A\left(a + \frac{x}{m}\right) -A(a),H\varphi\right) \phi_{\frac1m}\dx\right| \le \gamma\int_{\T^n}\left\|A\left(a + \frac{x}{m}\right) - A(a)\right\|\dx.
\]
By definition of Lebesgue point, the last term converges to $0$ as $m\to \infty$. This concludes the proof that $\lim_{m \to \infty}\lim_{j \to \infty} III_{k_j,\frac{1}{m}} = 0$.
\\
\\
\indent\fbox{Step 9: Conclusion}
\\
\\
Taking the limits in \eqref{equaz}, we achieve
$$
 \langle \nu_a, \det (\cdot)^\frac{1}{n-1}\rangle \int_{\T^n}  \varphi^{\frac{n}{n-1}}(x) \dx\leq \det(A(a))^\frac{1}{n-1}.
$$
By letting the cut-off function $\varphi$ converging to the characteristic function of the torus, we conclude the validity of \eqref{FM_ineq} almost everywhere.
\end{proof}

\begin{rem}\label{ext}
By analyzing the proof, it is moreover clear that one could slightly relax the assumptions of the Theorem. Indeed it would suffice to take a sequence $A_k \rightharpoonup A$ in $L^{\frac{n}{n - 1}}(\T^n)$  and $\dv(A_k) \overset{*}{\rightharpoonup} \dv(A)$ such that the sequence of Radon measures defined by $$\nu_k(E) = \int_E\det(A_k(x))^{\frac{1}{n - 1}} \dx, \quad \forall \text{ Borel } E\subset \T^n$$ weakly- $*$ converges in the sense of measures to a measure $\nu$ that is absolutely continuous with respect to the Lebesgue measure. In this case, calling $f$ the density of $\nu$ with respect to the Lebesgue measure on $\T^n$, one would prove that
\[
f(x) \le \det(A(x))^{\frac{1}{n -1}}\quad \text{ for a.e. }x \in \T^n,
\]
and conclude as in the proof of Theorem \eqref{t_main_sup}. In particular the sequence $\{A_k
\}_k$ does not need to be equibounded in $L^p$ for some $p > \frac{n}{n - 1}$.
 \end{rem}

As a simple consequence of the proof of Theorem \ref{t_main_sup}, we obtain the following 
\begin{Cor}
Let $p > \frac{n}{n - 1}$ and $\{A_k\}_k\subset X_p$ be such that $A_k \rightharpoonup A$ in $X_p$. Suppose further that $ \det (A_k)^\frac{1}{n-1}\rightharpoonup g$ in $L^1(\T^n)$. Then we have
$$
g(x)\leq \det (A(x))^\frac{1}{n-1},
$$
for almost every $x\in \T^n$.
\end{Cor}

\begin{proof}
Fix $\varphi \in C^\infty(\T^n)$ with $\varphi \ge 0$ and note that the sequence $\tilde A_k \doteq \varphi A_k$ is in $X_p$ for every $k$, and $\tilde A_k \rightharpoonup \varphi A$ in $X_p$. Using the hypothesis $\det^{\frac{1}{n - 1}}(A_k) \rightharpoonup g$ and applying Theorem \ref{t_main_sup} to the sequence $\tilde A_k$, we get
\begin{align*}
\int_{\T^n}g(x)\varphi^{\frac{n}{n - 1}}(x)\dx &= \lim_{k}\int_{\T^n}\det(A_k)^{\frac{1}{n - 1}}\varphi^{\frac{n}{n - 1}}(x)\dx  = \lim_k\D(\varphi A_k) \\
&= \lim_k\D(\tilde A_k) \le \limsup_k \D(\tilde A_k) \le \D(\varphi A) = \int_{\T^n}\det(A(x))^{\frac{1}{n - 1}}\varphi^{\frac{n}{n - 1}}(x)\dx.
\end{align*}
Since $\varphi$ was arbitrary, we conclude the proof.
\end{proof}

\section{The case $p \le \frac{n}{n - 1}$}\label{counterS}

In this section we prove the optimality of the assumptions of Theorem \ref{t_main_sup} and Remark \ref{ext}, by providing an explicit counterexample. In particular, we prove the following

\begin{prop}
For every $\varepsilon > 0$ and for every $x_0 \in \R^n$, there exists a sequence of matrix fields $A_k$ such that
\begin{enumerate}[(i)]
\item $A_k$ is compactly supported in $B_\varepsilon(x_0), \forall k \in \N$; \label{support}
\item $A_k\rightharpoonup 0$ in $L^{\frac{n}{n - 1}}(\R^n,\Sym^+(n))$ and strongly in $L^p(\R^n,\Sym^+(n))$, $\forall p < \frac{n}{n - 1}$; \label{convergence}
\item $\dv(A_k) \in \mathcal{M}(\T^n,\R^n)$, $\forall k$ and $\sup_{k \in \N}\|\dv(A_k)\|_{\mathcal{M}(\T^n,\R^n)} \le 1$; \label{divergence}
\item $\D(A_k) = \omega_n$, $\forall k$, so that in particular $\D(0) = 0 < \limsup_k\D(A_k) = \omega_n$;\label{determinant}
\end{enumerate}
\end{prop}

\begin{proof}
Fix a point $x_0 \in \R^n$ and consider
\[
f_k(x) \doteq 2^{k(n - 1)}\chi_{B_{2^{-k}}(x_0)}.
\]
Define $A_k(x) \doteq f_k(x)\id_n$. First we note that $\spt(A_k) \subset B_{2^{-k}}(x_0)$, $\forall k$, so that once $\varepsilon > 0$ is fixed, we can pick $k_0$ such that if $k \ge k_0$, then $(\ref{support})$ is fulfilled by choosing as a sequence $\{A_{k + k_0}\}_{k \in \N}$. Now note that the H\"older conjugate exponent of $\frac{n}{n - 1}$ is $n$. Hence, to see $(\ref{convergence})$, we compute for any $\varphi \in L^{n}(\R^n)$
\begin{equation}\label{1}
\left|\int_{\R^n}f_k(x)\varphi(x)\dx\right| = 2^{k(n-1)}\left|\int_{B_{2^{-k}(p)}}\varphi(x)\dx\right| \le \left(\int_{B_{2^{-k}(p)}}|\varphi|^n(x)\dx\right)^{\frac{1}{n}} \to 0, \text{ as }k\to \infty
\end{equation}
and, if $1 \le p < \frac{n}{n - 1}$,
\begin{equation}\label{1}
\|f_k\|_{L^{p}(\R^n)}^{p} = 2^{k(n-1)}2^{-k\frac{n}{p}} = 2^{-k(\frac{n}{p} - n + 1)}.
\end{equation}
The last expression converges to $0$ as $k \to \infty$ if $p < \frac{n}{n - 1}$, thus proving $(\ref{convergence})$. We turn to $(\ref{divergence})$. We observe that
\[
\dv(A_k) = Df_k,
\]
where $Df_k$ is the BV derivative of $f_k$. To compute it, we use the definition. For every $\Phi \in C^\infty(\T^n,\R^n)$ and for every $k \in \N$,
\[
\int_{\R^n} f_k(x)\dv(\Phi(x))\dx = 2^{k(n-1)}\int_{B_{2^{-k}}(x_0)}\dv(\Phi(x))\dx = 2^{k(n-1)}\int_{\partial B_{2^{-k}}(x_0)}(\Phi(z),\nu_k(z)) d\sigma(z),
\]
where $\nu_k(z) = \frac{z - x_0}{\|z - x_0\|}$ is the normal to $ \partial B_{2^{-k}}(x_0)$. The previous expression can be bounded with
\[
\left|\int_{\T^n} f_k(x)\dv(\Phi(x))\dx\right| \le \|\Phi\|_{C^0},
\]
hence also $(\ref{divergence})$ is fulfilled. Finally, we prove $(\ref{determinant})$:
\[
\int_{\R^n}\det(A_k(x))^{\frac{1}{n - 1}}\dx = \int_{\R^n}f_k^{\frac{n}{n - 1}}(x)\dx = \int_{B_{2^{-k}}}\left(2^{k(n-1)}\right)^{\frac{n}{n - 1}}\dx = \omega_n, \quad \forall k.
\]
This concludes the proof.
\end{proof}

\bibliographystyle{plain}
\bibliography{USC}

\end{document}